\documentclass[12pt]{article}

\usepackage{amsmath,amsthm,amssymb}
\usepackage{amscd,amsmath, amssymb, fancyhdr, url,epsfig, color, enumitem, tocloft}
\usepackage{epsfig,bbm}
\numberwithin{equation}{section}

\usepackage{tikz-cd}
\usepackage[all]{xy}
\usepackage{mathtext}
\usepackage[T1,T2A]{fontenc}
\usepackage[utf8]{inputenc}
\usepackage[russian,english]{babel}
\usepackage{amsfonts,longtable,amssymb,amsmath,latexsym, dsfont}
\usepackage{wrapfig}
\usepackage{xypic}
\usepackage{mathtools}
\usepackage{hyperref}

\title{Selfsimilar Hessian manifolds}
\author{Pavel Osipov\footnote{National Research University Higher School of Economics, Russian Federation.} \footnote{Pavel Osipov is 
		partially supported by the HSE University Basic
		Research Program, Simons Foundation, and by the contest “Young Russian
		Mathematics”.} 
}

\usepackage{amsthm, amsmath, amssymb, amsfonts, amscd, amscd, geometry}
\newsavebox{\ssa}
\renewcommand{\d}{\partial}
\newcommand{\e}{\varepsilon}

\newcommand{\dr}{\frac{\partial}{\partial t}}
\renewcommand{\d}{\partial}

\newcommand{\ds}{\frac{\partial}{\partial s}}

\newcommand{\R}{\mathbb R}
\renewcommand{\L}{\mathcal{L}}

\newcommand{\dxi}{\frac{\partial}{\partial x^i}}
\renewcommand{\dr}{\frac{\d}{\d r}}

\theoremstyle{definition}
\newtheorem{theorem}{Theorem}[section]
\newtheorem{lemma}[theorem]{Lemma}
\newtheorem{proposition}[theorem]{Proposition}

\newtheorem{cor}[theorem]{Corollary}
\newtheorem{defin}[theorem]{Definition}
\newtheorem{example}[theorem]{Example}

\renewcommand{\v}{\varphi}
\theoremstyle{remark}
\newtheorem{rem}{Remark}[section]
\usepackage{indentfirst}

\begin{document}
	
		\maketitle
		
		\begin{abstract}
		A selfsimiar manifold is a Riemannian manifold $\left(M,g\right)$ endowed with a homothetic vector field $\xi$. We characterize global selfsimilar manifolds and describe the structure of local selfsimilar manifolds. We prove that any selfsimilar manifold with a potential homothetic vector field is a conical Riemannian manifold or a Eucledean space. A radiant Hessian manifold is selfsimilar Hessian manifold $\left(M,\nabla,g,\xi\right)$ such that $\nabla\xi=\lambda \text{Id}$. We prove that any selfsimilar Hessian manifold with a potential homothetic vector field is locally isomorphic to a product radiant Hessian manifolds and describe the local structure of radiant selfsimialar Hessian manifolds. 
		\end{abstract}
	
	\tableofcontents

	\section{Introduction}
	
	We study Riemannian manifolds endowed with a homothetic vector fields. Precisely, a {\bfseries selfsimilar manifold} is a Riemannian manifold $(M,g)$ endowed with a vector field $\xi$ satisfying $\L_\xi g =2g$. Moreover, if the field $\xi$ is complete then we say that $\left(M,g,\xi\right)$ is a {\bfseries globally selfsimilar manifold}.  
	
	{\bfseries Riemannian cones} $\left(\hat M=M\times\R^{>0},\hat g=s^2 g_M+ds^2,\hat \xi=s\ds\right)$ are example of selfsimilar manifolds. Riemannian cones have important applications in supegravity (\cite{ACDM}, 
	\cite{ACM}, \cite{CDM}, \cite{VDMV}). It is well-known that a geodesically complete selfsimilar manifold is isometric to a Euclidean space. If the holonomy group of a Riemannian cone $\left(\hat M,\hat g,\hat \xi\right)$ is decomposable then $(M,g)$ has a constant curvature 1 (\cite{Ga}). Holonomy and geometry of pseudo-Riemannian cones are studied in \cite{ACGL} and \cite{ACL}. Selfsimilar Lorentzian manifolds are characterized in \cite{A}. A {\bfseries  conical Riemannian manifold} is a Riemannian manifold endowed with a vector field $\xi$ satisfying $D \xi=\text{Id}$, where $D$ is a Levi-Civita connection. Any conical Riemannian manifold is locally isometric to a Riemannian cone. Hence, any conical Riemannian manifold is a selfsimilar manifold.

		In Subsection \ref{s21}, we describe global selfsimilar manifolds. In Subsection \ref{s22}, we prove that any selfsimilar manifold is locally isomorphic to a global selfsimilar manifold. The results of subsections \ref{s21} and \ref{s22} are summarized by the following theorem. 
		
		\begin{theorem}\label{T_1}
			Any global selfsimilar manifold $(C,g,\xi)$ is isomorphic to one of the following: 
			\begin{itemize} 
				\item [(i)] $\left(\R^n, \sum_{i=1}^n\left(dx^i\right)^2,\rho+\eta\right)$,
				where $a\in \R$, $\rho=\sum_{i=1}^n x^i \dxi$ is a radiant vector field and $\eta\in \text{so}(n)$ is a Killing vector field.
				\item  [(ii)] $\left(\hat M=M\times \R^{>0},\hat g=s^2 g_M +s ds \cdot \alpha +  ds^2,s\ds \right)$, where
				$s$ is a coordinate on $\R^{>0}$, $g_M$ is a Riemannian metric on $M$, $\alpha$ a 1-form on $M$, and 
				$$
				g_M(X,X)+2\alpha(X)+ 1>0, \ \ \ \text{for any} \ \ \ X\in TM.
				$$
				The identification $(C,g,\xi)\simeq \left(\hat M,\hat g, s\ds\right)$ is defined by 
				$$
				s= \sqrt{g(\xi,\xi_)}, \ \ \ c \  =  \  \left(\gamma_c\cap \{s=1\}\right)\times s\  \in  \  M\times \R^{>0}, \ \ \ g_M= g|_{M\times 1}, \ \ \ \alpha = \iota_\xi g|_{M\times 1},
				$$ 
				where $c\in C$ and $\gamma_c$ is the integral curve of $\xi$ containing $c$.
			\end{itemize} 
			Any selfsimilar manifold is locally isomorphic to a global selfsimilar manifold. 
		\end{theorem}
		A selfsimilar manifold can be locally isomorphic to a global selfsimilar manifold from item (i) at a neighborhood of a point and isomorphic to one from item (ii) at a neighborhood of another point (see Example \ref{example with to parts}).

	We say that $(C,g,\xi)$ is a selfsimilar manifold with a {\bfseries potential homothetic vector field} if $\xi$ is locally defined as a gradient of a function. Denote If $\xi=\text{grad} \ f$ on a domain $U$ then $\iota_\xi g=0|_U=df$. Moreover, a form is closed if and only if it is locally exact. Therefore, the field $\xi$ is potential if and only if $d\iota_\xi g=0$.
	
		A Riemannian cone and a Euclidean space with a radiant vector field are selfsimilar manifolds with a potential vector field (see Example \ref{E1} and Example \ref{E2}). Actually, any global selfsimilar manifold with a radiant vector field belongs to examples above and these two examples describe the local geometry of selfimilar manifolds with potential vector fields. 
		
	The following theorem implies that any selfsimilar manifold with a potential homothetic vector field is locally isometric to a Eucledean space or a Riemannian cone.

\begin{theorem}\label{T1}\ 
	Let $\left(M,g,\xi\right)$ is a global selfsimilar manifold with a potential homothetic vector field.
	\begin{itemize} 
		\item [(i)] If $\xi$ vanishes at a point then $\left(M,g,\xi\right)$ is Euclidean space with a radiant vector field $\left(\R^n,\sum_{i=1}^n\left(dx^i\right)^2,\sum x^i\dxi\right)$.  
		\item [(ii)] If $\xi$ does not vanishes at any point then $\left(M,g,\xi\right)$ is a Riemannian cone $\left(\hat M,\hat g,\hat \xi\right)$. 
	\end{itemize}
	
\end{theorem}

If a global selfsimilar manifold $(C,g,\xi)$ is isometric to a Riemannian cone then the identification $C\simeq M\times\R^{>0}$ is defined as in Theorem \ref{T_1}.

In subsection \ref{s24}, we describe an example of a selfsimilar manifold without any potential homothetic vector field.

	Further, we work with Hessian manifolds. A {\bfseries flat affine manifold} is a differentiable manifold equipped with a flat torsion-free connection. Equivalently, it is a manifold equipped with an atlas such that all translation maps between charts are affine transformations (see \cite{FGH} or \cite{shima}). 
	A {\bfseries Hessian manifold} is a flat affine manifold with a Riemannian metric wich is locally equivalent to a Hessian of a function. 
		
	Hessian manifolds have many different application: in supersymmetry (\cite{CMMS}, \cite{CM}, \cite{AC}), in convex programming
	(\cite{N}, \cite{NN}), in the Monge-Ampère Equation (\cite{F1}, \cite{F2}, \cite{G}), in the WDVV equations (\cite{T}).
	
		A {\bfseries selfsimilar Hessian manifold} $(C,\nabla, g, \xi)$ is a Hessian manifold $(C,\nabla,g)$ endowed with a vector field $\xi$ such that $(C,g,\xi)$ is a selfsimilar manifold and the flow along $\xi$ preserves $\nabla$. If $\xi$ is complete then $(C,g,\xi)$ is called a {\bfseries global selfsimilar Hessian manifold}.
		
	A {\bfseries radiant manifold} $(C,\nabla, \rho)$ is a flat affine manifold $(C,\nabla)$ endowed with a {\bfseries radiant vector field} i.e. a field $\rho$ satisfying
	$$
	\nabla \rho =\text{Id}.
	$$
	Equivalently, it is a manifold equipped with an atlas such that all translation maps between charts are linear transformations (see e.g. \cite{Go}). In the corresponding flat affine coordinates we have 
	$$
	\rho=\sum x_i \dxi.
	$$

		We call a selfsimilar Hessian manifold $(C,\nabla, g, \xi)$ a {\bfseries radiant Hessian manifold} if and only if there exists a radiant vector field $\rho$ on $C$ and a constant $\lambda\ne 0,2$ such that $\xi=\lambda\rho$. There exists a radiant Hessian manifold with any $\lambda\ne 0,2$ (Corollary \ref{cor}).
		
		\begin{rem}
			The case $\lambda=2$ is studied in \cite{G-A}. In this case, the Hessian of a function $g=\text{Hess} \ \varphi$ satisfies 
			\begin{equation}\label{2}
			\iota_\xi g = 0,
			\end{equation}
			i.e. if $\lambda=2$ then $g$ can not be positive definite. A radiant affine manifold $\left(C,\nabla,\rho\right)$ endowed with a degenerate Hessian metric $g$ satisfying $\eqref{2}$ is called an {\bfseries extensive Hessian manifold}.  Extensive Hessian manifolds describe sets of states in equilibrium thermodynamics (see \cite{G-A} and \cite{W1}). In these models, coordinates are observables such as volumes, particle numbers, the internal energy or the entropy (we can choose only one of the last two observables for the independence of coordinates). Equation \eqref{2} is a coordinate-free form of the Gibbs-Duhem equation from thermodynamics (see e.g. \cite{W}). Equation \eqref{2} is equivalent to the existence of an Hessian potential $\varphi$ satisfying
			\begin{equation}\label{12}
			\L_\rho \varphi=\varphi  \ \ \ \Longleftrightarrow  \ \ \ \forall a\in \R^{>0}: \varphi(ax^1,\ldots,ax^n)=a\varphi (x^1,\ldots,x^n),
			\end{equation}
			(see \cite{G-A}). In equilibrium thermodynamics, $\varphi$ is either the entropy or the internal energy. In the first case, $g=\text{Hess} \varphi$ is called Ruppeiner metric (\cite{R1}, \cite{R2}) and, in the second case, Weinhold metric (\cite{W1}, \cite{W}). Entropy and internal energy are additive when systems are merged hence satisfy \eqref{12}.

		\end{rem}
		\begin{example}\label{EE}
			Let $\left(C_i,\nabla_i, g_i, \xi_i\right)$, be a collection of Hessian radiant manifolds, where ${1\le i\le k}$. Then $\left(\prod C_i,\bigoplus\nabla_i,\bigoplus g_i, \bigoplus\xi_i\right)$ is a selfsimilar Hessian manifold with a potential homothetic vector field (see Theorem \ref{l39}).
		\end{example}
	
		Actually any selfsimilar Hessian manifold is locally isomorphic to one from Example \ref{EE}.
		
		\begin{theorem} \label{l39}
			Let $\left(C,\nabla,\xi\right)$ be a selfsimilar Hessian manifold. Then $\xi$ is potential if and only if $\left(C,\nabla,\xi\right)$ is locally isomorphic to a direct product of radiant Hessian manifolds. Moreover, if the field $\xi$ is potential and vanishing at a point then $\left(C,\nabla,g,\xi\right)$ is a radiant Hessian manifold with a radiant vector field $\rho=\xi$. 
		\end{theorem}
	
	In Subsection \ref{s32}, we describe the local structure of radiant Hessian manifolds with a potential vector field. By Theorem \ref{l39}, this description provides a local structure of any selfsimilar manifold Hessian manifolds with a potential homothetic vector field.

\begin{theorem}\label{T4}
	Let $(C,\nabla,g,\xi)$ be an $n+1$-dimensional radiant selfsimilar Hessian manifold and $p\in C$.
	\begin{itemize}
		\item [(i)] Let $\xi_p=0$. Then the exists a flat local coordinate system $\left(x^1,\ldots x^{n+1}\right)$ at a neighborhood of $p$ such that 
		$$
		g=\sum_{i=1}^{n+1} \left(dx^i\right)^2 \ \ \text{and} \ \ 
		\xi=\sum_{i=1}^{n+1}\left(x^i\dxi\right).
		$$
		\item [(ii)]  Let $\xi_p \ne 0$. Then the exists a flat local coordinate system $\left(x^1,\ldots x^{n+1}\right)$ at a neighborhood of $p$ such that 
		$$
		x^{n+1}>0, \ \ \ \xi=\lambda\sum_{i=1}^{n+1}x^i\dxi \ \ \text{and} \ \ 
		g=\text{Hess} \left(\left(x^{n+1}\right)^{2\lambda^{-1}}\psi\right),
		$$ 
		where $\psi$ is a function constant along the radiant vector field such that
		\begin{gather}
				  \left(4\lambda^{-2}-2\lambda^{-1}\right)\psi>0  \label{1},\\   X^2  \psi>\max\left(\left(2-4\lambda^{-1}\right)X\left(\psi\right)-\left(4\lambda^{-2}-2\lambda^{-1}\right)\psi ,0 \right), \label{101}
		\end{gather}
		for any nonzero flat vector field $X=\sum_{i=1}^n b_i \dxi$, $b_i\in \R$.
		
	\end{itemize} 
\end{theorem}
	 
	 The system of inequalities \ref{1} and \ref{101} describes the condition for the metric $g=\text{Hess} \left(\left(x^{n+1}\right)^{2\lambda^{-1}}\psi\right)$ to be positive definite. Concretely,
	 $$
	 \begin{aligned}	
	 	g(\rho,\rho)>0 \ \ \  &\text{if and only if} \ \ \ \left(4\lambda^{-2}-2\lambda^{-1}\right)\psi>0, \\ 
	 	{g(X+\rho,X+\rho)>0}  \ \ \ &\text{if and only if}  \ \ \ X^2  \psi>\left(2-4\lambda^{-1}\right)X\left(\psi\right)-\left(4\lambda^{-2}-2\lambda^{-1}\right)\psi, \\ 
	 	g(X,X)>0 \ \ \ &\text{ if and only if} \ \ \ X^2(\psi)>0.
	 \end{aligned}
 	$$
 	Condition \ref{1} uniquely defines the sign of $\psi$ by $\lambda$. If $\lambda=2$, then condition \ref{1} cannot be satisfied. One can think of \ref{101} as a way to tell that the function $\psi$ is "strongly convex".

	\section{Selfsimilar manifolds}
	
	\begin{defin}
		A {\bfseries selfsimilar manifold} $(C,g,\xi)$ is a Riemannian manifold $(C,g)$ endowed with a field $\xi$ satisfying  
		$$
		\L_\xi  g = 2g.
		$$
		If $\xi$ is complete then the manifold is called a {\bfseries global selfsimilar manifold}.
	\end{defin}

	It follows from the definition that a global selfsimilar manifold is a Riemannian manifold endowed with a 1-parameter group of homothetic automorphims $\{\varphi_t\}$ such that $\varphi_t^* g= e^{2t} g$. The term "selfsimilar" is motivated by the fact that for any $\lambda\in \R^{>0}$ a global selfsimilar manifold $(C,g)$ is isometric to $(C,\lambda g)$. 

	\begin{example}
		Let $\left(\hat M=M\times\R^{>0},\hat g=s^2g_M+ ds^2\right)$ be a {\bfseries Riemannian cone} and $\xi=s\frac{\d}{\d s}$. Then $\left(\hat M,\hat g,\xi\right)$ is a globally selfsimilar manifold.
	\end{example}
	
	\begin{example}
		A {\bfseries conical Riemannian manifold} $(C,g,\xi)$ is a Riemannian manifold $(C,g)$ endowed with a nowhere vanishing vector field $\xi$ satisfying 
		$$
		D\xi =\text{Id},
		$$
		where $D$ is the Levi-Civita connection. The collection $(C,g,\xi)$ is locally isomorphic to a Riemannian cone ${\left(M\times\R^{>0},s^2g_M+ds^2,s\ds\right)}$ (see \cite{ACHK}). Therefore, any conical Riemannian manifold is a selfsimilar manifold. 
	\end{example}
\subsection{Global selfsimilar manifolds} \label{s21}
\subsubsection{The case of a nowhere vanishing homothetic vector field}\label{s211}
We will denote the symmetric product of tensors $\alpha$ and $\beta$ by $\alpha\cdot\beta$.

\begin{lemma}\label{L1}
	Let $s$ be a coordinate on $\R^{>0}$, $g_M$ a Riemannian metric on $M$, $\alpha$ a 1-form on $M$, and $f$ a positive definite function on $M$. Then the bilinear form
	$$
	g=s^2 g_M +s ds \cdot \alpha + f ds^2
	$$
	is positive if and only if  
	$$
	g_M(X,X)+2\alpha(X)+ f>0, \ \ \ \text{for any} \ \ \ X\in TM.
	$$
\end{lemma}
\begin{proof}
	
	Any tangent vector admits a form $as\ds$, $X$, or $as\ds+aX$, where $X\in TM$ and $a\in \R$. Since $f$ is positive definite, 
	$$
	g\left(as\ds,as\ds\right)=a^2s^2 f>0.
	$$
	Moreover,
	$$
	g(X,X)=s^2g_M(X,X)>0.
	$$
	Finally, we have
	$$
	g\left(as\ds+aX,as\ds+aX\right)=a^2s^2\left(g_M(X,X)+2\alpha(X)+ f\right).
	$$
	That is, 
	$$
	g\left(as\ds+aX,as\ds+aX\right)>0 \ \ \ \text{if and only if} \ \ \  g_M(X,X)+2\alpha(X)+ f>0.
	$$ Thus, 
	$$
	g=s^2 g_M +s ds \cdot \alpha + f ds^2
	$$
	is positive if and only if  
	$$
	g_M(X,X)+2\alpha(X)+ f>0, \ \ \ \text{for any} \ \ \ X\in TM.
	$$
\end{proof}

\begin{proposition}\label{p1}
	Let $g$ be a symmetric bilinear form on $M\times\R^{>0}$. Then $(M\times\R^{>0}, g,s\ds)$ is a selfsimilar manifold if and only if we have
	$$
	g=s^2 g_M +s ds \cdot \alpha + f ds^2,
	$$
	\begin{equation}\label{2.1}
	g_M(X,X)+2\alpha(X)+ f>0, \ \ \ \text{for any} \ \ \ X\in TM,
	\end{equation}
	where $s$ is a coordinate on $\R^{>0}$, $g_M$ a Riemannian metric on $M$, $\alpha$ a 1-form on $M$, and $f$ a positive definite function on $M$.
	
\end{proposition}

\begin{proof}
	The bilinear form 
	$
	g=s^2 g_M +s ds \cdot \alpha + f ds^2
	$
	satisfies $\L_{s\ds} g = 2g$. According to Lemma \ref{L1}, $g$ is positive definite if and only if the condition \eqref{2.1} holds. 
	
	Let $\left(M\times\R^{>0}, g,s\ds\right)$ be a selfsimilar manifold.  The condition $\L_{s\ds} g =2g$ is equivalent to the condition $\lambda_t g = t^2g$, where the map $\lambda_t: M\times \R^{>0} \to M \times \R^{>0}$ defined by the rule  $\lambda_t (x\times s)=x\times ts$. Let $\pi: M\times\R^{>0}\to M$ be the projection. Then the metric $g$ lies in 
	$$
	\text{Sym}^2 \left(\pi^*TM \oplus \ker \pi\right)^* =\text{Sym}^2 \left(\pi^*TM\right)^* \oplus \text{Sym}^2 \left(\pi^*TM \otimes \ker \pi \right)^* \oplus \text{Sym}^2 \left(\ker \pi \right)^*.
	$$

	Let $\pi_{2, 0}$, $\pi_{1,1}$, and $\pi_{0,2}$ be the projections of $\text{Sym}^2 \left(\pi^*TM \oplus \ker \pi\right)^*$ on the three summands from the decomposition above. Define $g|_{M\times 1} =g_M$. For any $X_1, X_2 \in \pi^*TM$ we have
	$$
	g(X_1, X_2)(m\times s)=\lambda_s^* g(X_1, X_2)(m\times 1)=s^2 g(X_1, X_2)(m\times 1).
	$$
	Thus,
	\begin{equation}\label{2.2}
	\pi_{2,0} g= s^2 g_M.
	\end{equation}
	
	For any $X \in \pi^*TM$ and $\frac{\partial}{\partial s} \in \ker \pi$, since 
	$$
	{\lambda_s}_* \left( s^{-1}\frac{\partial}{\partial s}\right) =\frac{\partial}{\partial s},
	$$ 
	we have
	\begin{equation}\label{2.3}
	g\left(X, \frac{\partial}{\partial s}\right)(m\times s)=\lambda_s^* g\left(X, s^{-1}\frac{\partial}{\partial s}\right)(m\times 1)=s g\left(X,\frac{\partial}{\partial s} \right)(m\times 1). 
	\end{equation}
	Define
	$$
	\alpha := \left.\left(\iota _{\frac{\partial}{\partial s}} g\right)\right|_{M\times 1}.
	$$
	Then, by \eqref{2.3}, we get
	$$
	g\left(X, \frac{\partial}{\partial s}\right)(m\times s)=s\alpha (X)=
	s ds \cdot\alpha \left(X, \frac{\partial}{\partial s}\right).
	$$
	Hence, 
	\begin{equation}\label{2.4}
	\pi_{1,1} g =s ds \cdot\alpha ,
	\end{equation}
	
	Finally, we have
	$$
	g\left(\frac{\partial}{\partial s},\frac{\partial}{\partial s}\right)(m,s)=\lambda_s^* 	g\left(s^{-1}\frac{\partial}{\partial s},s^{-1}\frac{\partial}{\partial s}\right)(m,1)=	g\left(\frac{\partial}{\partial s},\frac{\partial}{\partial s}\right)(m,1).
	$$
	Therefore, we get
	\begin{equation}\label{2.5}
	\pi_{0,2} g= f ds^2,
	\end{equation}
	where the function 
	$$
	f(m)=g\left(\frac{\partial}{\partial s},\frac{\partial}{\partial s}\right)(m,1)
	$$ is a function on $M$.
	
	Combining \eqref{2.2}, \eqref{2.4}, and \eqref{2.5}, we get 
	$$
	g=s^2 g_M +s  ds \cdot \alpha + f ds^2.
	$$
	By Lemma \ref{L1}, $g$ is positive definite if and only if the condition \eqref{2.1} holds.
\end{proof}
\begin{theorem}\label{global selfsimilar with nowhere vanishing}
	Any global selfsimilar manifold $(C,g,\xi)$ with a nowhere vanishing homothetic vector field is isomorphic to $$\left(\hat M=M\times \R^{>0},\hat g=s^2 g_M +s  ds \cdot \alpha +  ds^2,s\ds \right),$$ where
	$s$ is a coordinate on $\R^{>0}$, $g_M$ a Riemannian metric on $M$, $\alpha$ a 1-form on $M$, and 
	$$
	g_M(X,X)+2\alpha(X)+ 1>0, \ \ \ \text{for any} \ \ \ X\in TM.
	$$
	The identification $(C,g,\xi)\simeq \left(\hat M,\hat g, s\ds\right)$ is defined by 
	$$
	s= \sqrt{g(\xi,\xi_)}, \ \ \ c \  =  \  \left(\gamma_c\cap \{r=1\}\right)\times s\  \in  \  M\times \R^{>0}, \ \ \ g_M= g|_{M\times 1}, \ \ \ \alpha = \iota_\xi g|_{M\times 1},
	$$ 
	where $c\in C$ and $\gamma_c$ is the integral curve of $\xi$ containing $c$.
	
\end{theorem}

\begin{proof}
	The vector field $\xi$ does not vanish at any point of $C$. The Lie derivative of a metric is defined by 
	$$
	(\L_X g )(Y,Z)=\L_X (g(Y,Z))-g([X, Y], \ Z)-g(Y, \ [X, Z]).
	$$
	Therefore,
	$$
	\L_\xi \ (g(\xi, \xi))= (\L_\xi \ g)(\xi, \xi) = 2g(\xi, \xi)>0.
	$$
	Thus, the function $s:=|\xi|=\sqrt {g(\xi, \xi)}$ strictly increases along $\xi$. For any point $p\in C$ denote the integral curve of $\xi$ containing a point $p$ by $\gamma_p$. Since $\xi_p \ne 0$, the curve $\gamma_p$ is non-degenerated. Then any integral curve $\gamma_p$ is diffeomorphic to $\R$ because there exists a strictly increasing function along $\gamma$ ($s$ strictly increases along $
	\gamma$). Let $u_p$ be a coordinate on $\gamma_p$ such that 
	$$
	\frac{\partial}{\partial u_p}=\xi.
	$$  
	Then $s=|\xi|$ satisfies 
	$$
	\frac{\partial}{\partial u_p} s^2 = \L_\xi (g(\xi,\xi))=2g(\xi,\xi) = 2s^2.
	$$
	Therefore, 
	$$
	s=ae^{u_p},
	$$
	where $a\in \R$ is a constant. Moreover $a\in \R^{>0}$ because $s$ is positive defined. Then $s$ identifies $\gamma$ with $\R^{>0}$ and 
	$$
	\xi=\frac{\partial}{\partial u_p}=\frac{\partial}{\partial \ln (a^{-1}s)}=\left( \frac{\partial\ln (a^{-1}s)}{\partial s}\right)^{-1} \frac{\partial}{\partial s}=s\frac{\partial}{\partial s}.
	$$
	Denote by
	$$
	M:=\{p\in C|\ s(p)=1\}
	$$
	the level set which is a smooth submanifold of a codimension $1$. Since for any $p\in C$, the map $s|_{\gamma_p} : \gamma_p \to \R^{>0}$ is an isomorphism, $\gamma_p$ has an unique intersection with $M$. Define
	$$
	\alpha: C \to M \times \R^{>0}
	$$
	by the rule 
	$$
	\alpha(p)=(\gamma_p\cap M) \times s(p).
	$$
	For any $p\in C$ the restriction $s|_{\gamma_p}$ is an isomorphism.
	Moreover, $$C=\bigcup_{p\in C} \gamma_p=\bigcup_{p\in M} \gamma_p.$$ Therefore, $\alpha$ is an isomorphism.

	Thus, we have $C=M\times \R^{>0}$ and $\xi=s\frac{\partial}{\partial s}$, where $s$ is a coordinate on $M$. That is, $(M\times\R^{>0},g,s\ds)$ is a selfsimilar Riemannian manifold. Thus, $g$ satisfies the conditions of Proposition \ref{p1}.
	
\end{proof}

\subsubsection{The case of a vanishing at a point homothetic vector field}\label{s212}

\begin{lemma}\label{L2}
	Let $(C,g,\xi)$ be a global selfsimilar manifold and $\Phi_t=\exp(-t\xi)$ the flow along the vector field $-\xi$. Suppose that the field $\xi$ vanishes at a point $p\in C$. Then for any $q\in C$ we have
	$\underset{t\to\infty}\lim\Phi_t(q)=p.$
\end{lemma}

\begin{proof}
	The vector field $\xi$ vanishes at a point $p$. Let $q\in C$, $\gamma$ be a path between the points $p$ and $q$, $l$ the length of $\gamma$, $\Phi_t$ the flow along the field $\xi$ as above. For any $t\in \R^{>0}$,  $\Phi_t(\gamma)$ is a path between $\Phi_t(q)$ and $\Phi_t(p)=p$. Moreover, the length of  $\Phi_t(\gamma)$ equals $e^{-2t}l$. We have $\underset{t\to\infty}\lim e^{-2t}l=0$. Hence, $\underset{t\to\infty}\lim\Phi_t(q)=p$.
\end{proof}

\begin{proposition}\label{p2}
	Let $(C,g,\xi)$ be a global selfsimilar manifold and the homothetic field $\xi$ vanishes at a point $p$. Then $(C,g)$ is a flat Riemannian manifold. 
\end{proposition}

\begin{proof}
	Let $\Phi_t=\exp(-t\xi)$ be as in Lemma \ref{L2}, $\Theta$ the Riemannian curvature tensor,
	$$
	K(u,v)=\frac{g(\Theta_v(u,v),u)}{g(u,u)g(v,v)-(g(u,v))^2}
	$$ 
	the sectional curvature, and
	$$
	k(q)=\max_{u,v\in T_q M} \left|K(u,v)\right|.
	$$
	The latter is defined because $K(u,v)$ depends only on directions of $u$ and $v$ but not on the lengths.
	Suppose that there exists $q\in C$ such that $k(q)>0$. 
	
	Since $\Phi_t^* g=e^{-2t}g$, the Riemannian manifold $(\Phi_t C, e^{2t} g)$ 
	is isometric to $(C,g)$. Therefore, 
	$$
	k(\Phi_t q)=e^{2t} k(q).
	$$ 
	The function $k$ is continuous. Moreover, we have
	$$
	\lim_{t\to \infty} \Phi_t q =p
	$$  
	according to Lemma \ref{L2}.
	Hence,
	$$
	k(p)=\lim_{t\to \infty} k(\Phi_t q)=\lim_{t\to \infty}e^{2t} k(q) =\infty.
	$$
	But $k(p)$ is finitely defined. Thus the sectional curvature is equal to zero on $C$. Therefore, $(C,g)$ is a flat Riemannian manifold.
	
\end{proof}

\begin{theorem}\label{global selfsimilar manifold with vanishing}
	Any global selfsimilar manifold $(C,g,\xi)$ is isomorphic to $\left(\R^n, \sum_{i=1}^n\left(dx^i\right)^2,\rho+\eta\right)$,
		where $a\in \R$, $\rho=\sum_{i=1}^n x^i \dxi$ is an Euler vector field, $
		\eta \in \text{so}(n)$ is a Killing vector field.
\end{theorem}

\begin{proof}
	According to Proposition \ref{p2}, $(C,g)$ is a flat Riemannian manifold.
	
	The Riemannian manifold $C$ contains a geodesic ball $B_p$ with center $p$ and radius $\e$. Since $(C,g)$ is flat, $B_p$ is isometric to the ball $D^\e\subset \R^n$ of radius $\e$. Then $(\Phi_{t} B_p, g)$ is isometric to $(B_p, e^2t g)$ i.e.  $(\Phi_{t} B, g)$ is isometric to the ball $D^{e^{2t}\e}\subset \R^n$ of radius $e^{2t}\e$. Since $V=C$, we have
	$$
	\lim_{t \to\infty} \Phi_t B_p=\bigcup_{t\in \R^{>0}} \Phi_t B_p=C.
	$$   
	On the other hand,
	$$
	\lim_{t \to\infty} \Phi_t B_p=\lim_{t \to\infty} D^{e^{2t}\e}=\R^n	.
	$$
	Therefore, 
	$
	C= \R^n.
	$
	Since
	$$
	\L_{\rho} g=2g=\L_\xi g,
	$$  
	the field $\eta:=\xi-\rho$ is killing.
\end{proof}

\subsection{A local structure of selfsimilar manifolds}\label{s22}

\subsubsection{The case of a point where a homothetic vector field does not vanish}\label{s221}

\begin{lemma}\label{L3}
	Let $M$ be a manifold, $U$ an open subset of $M\times \R^{>0}$, $s$ a coordinate on $\R^{>0}$, $\pi: M\times \R^{>0} \to M$ a natural projection, and $g$ a Riemannian metric on $U$. For any $t\in \R^{>0}$ consider a map $\lambda_t: M\times \R^{>0} \to M \times \R^{>0}$ defined by  $\lambda_t (x\times s)=x\times ts$. Let for any $x\in U$ and $t\in \R^{>0}$ satisfying $\lambda_t x \in U$ we have 
	$$
	(\lambda_t^* g )(x) =t^2 g(x).
	$$
	Then the metric $g$ on $U$ can be extended to a selfsimilar metric $\tilde g$ on $(\pi (U)\cap M)\times \R^{>0}$. 
\end{lemma}

\begin{proof}
	For any $x\in (\pi (U)\cap M)\times \R^{>0}$ there exists $q\in \R^{>0}$ such that $\lambda_t(x) \in U$. Define $\tilde{g}(x)$ by 
	$$
	\tilde{g}(x)(X,Y)=t^{-2}\lambda_t^*g(x)(X,Y)= t^{-2}g(\lambda_t x)({\lambda_t}_* X,{\lambda_t}_* Y),
	$$
	for any $X,Y \in T_x (M\times\R^{>0})$. We need to check that the value $\tilde{g}(x)$ does not depend on the choice of $t$. Let $t'\ne q$ satisfies $\lambda_{t'} x \in U$. Then, by the conditions of the lemma,
	$$
	g(\lambda_{t'} x)({\lambda_{t'}}_* X,{\lambda_{t'}}_* Y)=\lambda_\frac{t'}{t}^* g(\lambda_t x)({\lambda_t}_* X,{\lambda_t}_* Y)=\left(\frac{t'}{t}\right)^{2}g(\lambda_t x)({\lambda_t}_* X,{\lambda_t}_* Y).
	$$
	Therefore, 
	$$
	{t'}^{-2}g(\lambda_{t'} x)({\lambda_{t'}}_* X,{\lambda_{t'}}_* Y)=t^{-2}g(\lambda_t x)({\lambda_t}_* X,{\lambda_t}_* Y).
	$$
	Thus, the value $\tilde{g}(x)$ is correctly defined. Moreover, $\tilde{g}(x)$ is obviously selfsimilar. 
	
\end{proof}			

\begin{theorem}\label{local structure not vanish}
	Any selfsimilar manifold is locally isomorphic to a global selfsimilar manifold near a point where the homothetic vector field does not vanish. 
\end{theorem} 

\begin{proof}
	Let $(C,g,\xi)$ be a selfsimilar manifold. As in the proof of Theorem \ref{global selfsimilar with nowhere vanishing}, for any $p\in C$ define
	$$
	s(p)=|\xi|(p)=\sqrt{g(\xi,\xi)}\ (p).
	$$
	Consider a point $x\in C$. The level set
	$$
	M=\{p\in C \ |\ s(p)=s(x)\}
	$$
	is a smooth submanifold by the same argument as in the proof of Theorem \ref{global selfsimilar with nowhere vanishing}. A manifold $C$ contains a tubular neighborhood $U$ with respect to the field $\xi$. By the definition, $U$ is diffeomorphic to an open subset $W\subset M \times \R$. The diffeomorphism $\alpha: W \to U$ is defined by the rule
	$$
	\alpha : p\times u \to \exp_p (u\xi). 
	$$
	As in the proof of Theorem \ref{global selfsimilar with nowhere vanishing}, by solving a differential equation on $r$ along integral curves we get 
	$$
	s=a e^u
	$$
	and 
	$$
	\xi= s\frac{\partial}{\partial s}
	$$
	in $U$. Thus, $U$ is a subset $M\times \R^{>0}$ and 
	$
	\xi= s\frac{\partial}{\partial s},
	$
	where $s$ is a coordinate on $\R^{>0}$. For any $x\in U$ and $t\in \R^{>0}$ satisfying $\lambda_t x \in U$ we have 
	$$
	(\lambda_t^* g )(x) =t^2 g.
	$$
	By Lemma \ref{L3}, the metric $g$ on $U$ can be extended to a selfsimilar metric $\tilde g$ on ${(\pi (U)\cap M)\times \R^{>0}}$. Therefore,  $U$ is isomorphic to an open subset of the selfsimilar cone $\left((\pi (U)\cap M)\times \R^{>0},\tilde{g},s\frac{\partial}{\partial s}\right)$.
	
\end{proof}

\subsubsection{The case a point where a homothetic vector field vanishes}\label{s222}

\begin{lemma}\label{L4}
	Let $(C,g)$ be a Riemannian manifold and $K\subset C$ a compact subset. Then there exists $\e\in \R^{>0}$ such that
	$$
	K^\e=\{p\in C\ | \ \text{dist}(p,K)\le \e\}
	$$ 
	is a compact submanifold with a boundary.  
\end{lemma}
\begin{proof}
	For any $q\in K$, the manifold $C$ contains a closed geodesic ball $\overline{B_q}$ with center in $q$ and radius $\e_q$. Then $\{B_q\}$ is an open covering if $K$. Let $B_{q_1},\ldots,B_{q_k}$ be a finite subcovering. Set
	$$
	N=\bigcup_{1\le i\le k}{B_{q_i}} .
	$$
	Then
	$$
	\overline N=\bigcup_{1\le i\le k} \overline {B_{q_i}}=\overline{\bigcup_{1\le i \le k}  B_{q_i}} 
	$$ 
	is compact. Hence, the boundary $\d N$ is compact. Consider two cases:

	i) $\d N=\emptyset$. Then $\overline  N$ is closed and open subset of $C$. Therefore $\overline N=C$ and for any $\varepsilon \in\R^{>0}$ the set 
	$$
	K^\e=\{p\in M\ | \ \text{dist}(p,K)\le \e\}
	$$
	is compact because $C=\overline N$ is compact and $\text{dist}(*,K)$ is a continuous function.

	ii) $\d N \ne \emptyset$. Then the exists a point $p\in \d N$ such that 
	$$
	\left.\left.\min_{x\in \d N}\right(\text{dist}(x,K)\right)=\text{dist}(p,K). 
	$$
	Suppose $\text{dist}(p,K)=0$, i.e. $p\in K\subset N$. The set $N$ is open hence $N\cap\d N=\emptyset$. But $p \in \d N$. Therefore, $\varepsilon:=\text{dist}(p,K)>0$. Then
	$$
	K^\e=\{p\in M\ | \ \text{dist}(p,K)\le \e\}=\{p\in N\ | \ \text{dist}(p,K)\le \e\}
	$$
	is compact because $N$ is compact and $\text{dist}(*,K)$ is a continuous function. 
\end{proof}

\begin{lemma}\label{L5}
	Let $(C,g,\xi)$ be a selfsimilar manifold and $\Phi_t=\exp(-t\xi)$ the flow along the vector field $-\xi$. Suppose that the vector field $\xi$ vanishes at a point $p\in C$. Let $V$ be the set of point $q$ satisfying: 
	\begin{itemize}
		\item [(i)] $\Phi_t(q)$ is defined for any $t\in[0,\infty)$.
		\item [(ii)] $\underset{t\to\infty}\lim\Phi_t(q)=p. $
	\end{itemize}
	Then $V$ is an open subset of $C$.
\end{lemma}

\begin{proof}
	A Riemannian manifold $C$ contains a geodesic ball $B_p^\epsilon$ with center $p$ and radius $\epsilon$. Let us show that $B_p^\epsilon\subset V$. The flow $\Phi_t$ fixes the point $p$ and decreases distances.
	Therefore, $\Phi_t$ does not increase a distance between the point $p$ and an image of $q\in V$. Hence, the flow does not decrease the distance between the bound $\d B_p^\epsilon$ and the image of $q$. Thus, the flow $\Phi_t$ is defined for $t\in [0,\infty)$. Let $\gamma$ be a path between $q$ and $p$ lying into $B_p^\epsilon$ and $l$ be a length of $\gamma$. Then $\Phi_t(\gamma)$ is a curve connecting $\Phi_t (\gamma)$ with $p$ and the length of $\Phi_t (\gamma)$ is $e^{-2t}l$. Since
	$$
	\lim_{t \to \infty} e^{-2t}l=0,
	$$
	we have
	$$
	\lim_{t \to \infty} \Phi_t (\gamma) =p.
	$$
	
	Let $v\in V$ and $\gamma$ be a path between $v$ and $p$ along $\xi$. By Lemma \ref{L4}, there exists $\e \in \R^{>0}$ such that
	$$
	\gamma^\e_v = \{x\in M\ | \ \text{dist}(x,\gamma)\le \e \}
	$$
	is a compact submanifold with a boundary. As above, the flow $\Phi_t$ does not increase a distance between the path $\gamma$ and an image of $q\in \gamma^\e_v$.  Hence, for any $q\in \gamma^\e_v$ the flow does not increase the distance between the bound $\d \gamma^\e_v$ and the image of $q$. We have the flow along a vector field on a compact manifold with a boundary not decreasing the distance to the boundary.  Therefore, the flow is defined for any point in the interior and for all positive times. Moreover, $\underset{t\to\infty}\lim\Phi_t(q)=p $ by the same argument as above. Hence, $\gamma^\e_v$ lies in $V$. 
	
	For any point $v\in V$ we can construct an open neighborhood $\gamma^\e_v$ lying in $V$. Thus, $V$ is open.

\end{proof}

\begin{theorem}\label{local structure vanish}
		Let $(C,g,\xi)$ be a selfsimilar manifold and $\Phi_t=\exp(-t\xi)$ the flow along the vector field $-\xi$. Suppose that the field $\xi$ vanishes at a point $p\in C$. Let $V$ be a set of point $q$ satisfying: 
	\begin{itemize}
		\item [(i)] $\Phi_t(q)$ is defined for any $t\in[0,\infty)$.
		\item [(ii)] $\underset{t\to\infty}\lim\Phi_t(q)=p. $
	\end{itemize}
	Then $V$ is an open neighborhood of $p$ and $(V,g,\xi)$ is isomorphic to a part of global selfsimilar manifold. 
\end{theorem}

\begin{proof}
 	According to Lemma \ref{L5}, $V$ is open, There $\left(V,g|_V\right)$ is a Riemannian manifold. Hence, $V$ contains a closed geodesic ball $B$ with center in $p$. Let $S^{n-1}=\d B$ be the corresponding sphere. The flow $\Phi$ decrease distances to $p$. Therefore, for any $q\in V$ there exists an unique $s_q\in \R$ such that $\Phi_{s_q} q \in S^{n-1}$.
 	Define $$\alpha: V \setminus p \to S^{n-1} \times \R \ \ \ \alpha(q)= \Phi_{s_q} q\times s_q.$$
 	The map $\alpha $ is a continuous embedding. Then the metric $\alpha_* g$ is correctly defined on $\alpha(V)\subset S^{n-1}\times \R$. 
 	Moreover, $\alpha_* \xi = \ds$ therefore $\left(\alpha(V),\pi_*g,\ds\right)$
 	is a selfsimilar manifold. By Lemma \ref{L3}, $\alpha_* g$ can be extended to a selfsimilar metric on $S^{n-1}\times\R$. The gluing 
 	$$
 	C:=\left. V \sqcup\left( S^{n-1}\times\R \right)\right/_{\alpha(q)\sim q}\simeq \R^n
 	$$
 	is a global selfsimilar manifold.
\end{proof}
	Combining theorems \ref{global selfsimilar manifold with vanishing}, \ref{global selfsimilar with nowhere vanishing}, \ref{local structure not vanish}, and \ref{local structure vanish} we get Theorem \ref{T_1}. 

\begin{example}\label{example with to parts}
	Let $L= \{(1,y)\subset \R^2|y\in [-1,1]\}$, $C=\R^2 \setminus L$,
	$W=\{A\in C| 0p\cap L \ne \emptyset\}$, where $0p$ is a line segment. Let $\varphi$ be a smooth function on $\R$ such that 	
	\begin{align*}
	\varphi(x)=1 \  &\text{for} \ x\in (-\infty, -1] \cup [1,\infty),\\
	\varphi(x)>1 \ &\text{for}\ x\in (-1,1). 
	\end{align*}
	Define
	$$
	g=\begin{cases}
	dx^2 +dy^2 \ &\text{on} \ C \setminus W; \\
	g=\varphi\left(\frac{y}{x}\right) dx^2+dy^2\  &\text{on}\  W.
	\end{cases}
	$$
	Then $\left(C,g,\xi=x\frac{\d}{\d x}+y\frac{\d}{\d y}\right)$ is a selfsimilar manifold. Let $V$ be as in Theorem \ref{local structure vanish}. Then $V=C\setminus W$ be the maximal open flat Riemannian submanifold of $(C,g)$. The field $\xi$ vanishes at a point but $(C,g)$ is not flat. 
\end{example}

\subsection{Selfsimilar manifolds with a potential homothetic vector fields}\label{s23}

\begin{defin}
We say that $(C,g,\xi)$ is a selfsimilar manifold with a {\bfseries potential vector field} if $\xi$ is locally defined as a gradient of a function. If $\xi=\text{grad} \ f$ on a domain $U$ then $\iota_\xi g|_U=df$. Moreover, a form is closed if and only if it is locally exact. Therefore, the field $\xi$ is potential if and only if $d\iota_\xi g=0.$
\end{defin}
\begin{example}\label{E1}
	Let $\left(M\times\R^{>0},s^2g_M+ds^2\right)$ be a Riemannian cone then $s\ds$ is a potential vector field. Actually,
	$$
	d\theta= d\iota_{s\ds} \left(s^2g_M+ds^2\right)=d(sds)=0.	
	$$ 
\end{example}

\begin{example}\label{E2}

A radiant vector field $\rho=\sum_{i=1}^n x^i\dxi$ on a Euclidean space $\left(\R^n,g=\sum_{i=1}^n\left(dx^i\right)^2\right)$ is a potential homothetic vector field. Actually, 
$$
d\iota_\rho g= d\left(\sum_{i=1}^n x^i dx^i\right)=0.
$$ 
\end{example}

\begin{proposition}\label{P}
	Let 
	$
	\left(M\times\R^{>0},g=s^2 g_M +sds \cdot \alpha + f ds^2,s\ds\right)
	$
	be a selfsimilar manifold (see Proposition \ref{p1}). Then the field $s\ds$ is potential if and only if 
	$$
	2\alpha+df=0.
	$$
\end{proposition}	
\begin{proof}
	We have
	$$
	d\iota_{s\frac{\partial}{\partial s}} g=d(s^2\alpha +sf ds)=(2\alpha +df)sds +s^2 d\alpha.
	$$
	Therefore, if $df=-2\alpha$ then $d\alpha=0$ and $d\iota_{s\frac{\partial}{\partial s}} g=0$. Hence, $df=-2\alpha$ then  the field $s\ds$ is potential. 
	
	Let  $s\ds$ is potential. Then
	\begin{equation}\label{2.8}
	(2\alpha +df)sds +s^2 d\alpha=d\iota_{\frac{\partial}{\partial s}}g=0.
	\end{equation}
	Since $f$ and $\alpha$ do not depend on $s$, we have
	$$
	\iota_{\frac{\partial}{\partial s}} df=0, \ \ \ \iota_{\frac{\partial}{\partial s}}d\alpha=0, \ \  \ \text{and} \ \   \iota_{\frac{\partial}{\partial s}}\alpha=0.
	$$ 
	Combining this with \eqref{2.8}, we get
	$$
	0= \iota_{\frac{\partial}{\partial s}}((2\alpha +df)sds +s^2 d\alpha)=\left(2\iota_{\frac{\partial}{\partial s}}\alpha+ \iota_{\frac{\partial}{\partial s}}df\right)sds+(2\alpha+df)s+s^2\iota_{\frac{\partial}{\partial s}}d\alpha=(2\alpha+df)s.
	$$
	Hence, 
	$$
	(2\alpha+df)s	=0,
	$$
	and, since $s$ is positive definite, 
	$$
	2\alpha+df=0.
	$$
	
\end{proof}

\begin{proof}[Proof of Theorem \ref{T1}]

	(i) According to Theorem \ref{T_1} a global selfsimilar manifold with a vanishing vector field is isomorphic to a collecting $\left(\R^n, \sum_{i=1}^n\left(dx^i\right)^2,\rho+\eta\right)$, where $a\in \R$, $\rho=\sum_{i=1}^n x^i \dxi$ is a radiant vector field and $\eta$ be a Killing vector field. We have 
	$
	d\iota_{\rho} g=0.
	$
	Therefore, $\rho+\eta$ is potential if and only if $d\iota_\eta=0$. In a proper ortonormal coordinate system $\left(\tilde x^1,\ldots,\tilde x^n\right)$ the Killing vector field $\eta$ admits a form 	$
	\eta = {\sum_{i=1}^{\left[n/2\right]} a_i\left(\tilde x^{2i+1}\frac{\partial}{\partial \tilde x^{2i}}
		- \tilde x^{2i}\frac{\partial}{\partial \tilde x^{2i+1}}\right)}$, $a_i\in\R$ (any Killing cetor field admits such form). Thus,
	$$
	0=\iota_\eta g=\sum_{i=1}^{\left[n/2\right]} a_i\left(d\tilde x^{2i+1}\wedge d\tilde x^{2i}-d\tilde x^{2i}\wedge d\tilde x^{2i+1}\right)=\iota_\eta g=\sum_{i=1}^{\left[n/2\right]} 2a_id\tilde x^{2i+1}\wedge d\tilde x^{2i}.
	$$
	Thus, for all $i$ we have $a_i=0$. That is, the field $\rho+\eta$ is potential if and only if $\eta=0$.

	(ii) According to Theorem \ref{T_1} a global selfsimilar manifold with a nowhere vanishing vector field is isomorphic to a collection  $\left(\hat M=M\times \R^{>0},\hat g=s^2 g_M +s ds \cdot \alpha +  ds^2,s\ds \right)$. By Proposition \ref{P}, the field $s\ds$ is potential if and only if $\alpha=0$. Thus, a global selfsimilar manifold with a nowhere vanishing potential vector field is a Riemannian cone.

\end{proof}

	\begin{proposition}\label{pp}
		Let $\left(C,g,\xi\right)$ is a selfsimilar manifold with a potential homothetic vector field and $\varphi=\frac{1}{2}g(\xi,\xi)$. Then we have $$\xi=\text{grad} \ \varphi.$$
	\end{proposition}
	
	\begin{proof}
		According to Theorem \ref{T1}, a selfsimilar manifold $\left(C,g,\xi\right)$ with a potential homothetic vector field is locally isomorphic to a selfsimilar manifold $\left(M\times \R^{>0},s^2g_M+sds\cdot\alpha +ds^2,s\ds\right)$ or $\left(\R^n, \sum_{i=1}^n\left(dx^i\right)^2,\rho=\sum_{i=1}^nx^i\dxi\right)$. In the first case,
		$$
		\iota_\xi g = sds=\frac{1}{2}ds^2=d\varphi.
		$$
		In the second case,
		$$
		\iota_\xi g=\sum_{i=1}^n x^idx^i=d\varphi
		$$
		Thus, $\xi=\text{grad} \ \varphi$. 
		
	\end{proof}

	\subsection{A selfsimilar manifold without any potential homothetic vector fields}\label{s24}

	\begin{example}
		Consider a collection $\left(C=S^1\times \R^{>0},g=s^2d\varphi^2+sds\cdot d\varphi+ds^2\right)$, where $\varphi$ and $s$ are coordinates on $S^1$ and $\R^{>0}$. Any tangent vector admits a form $a\frac{\d}{\d\varphi}+b\frac{\d}{\d\varphi}$. We have 
		$$
		d\varphi^2\left(a\frac{\d}{\d\varphi},a\frac{\d}{\d\varphi}\right)
		+2d\varphi\left(a\frac{\d}{\d\varphi}\right)+1=a^2+2a+1>0. 
		$$
		By Proposition \ref{p1}, $\left(C,g,s\ds\right)$ is a selfsimilar manifold. 
		Any homothetic vector field ${\xi=a\ds+b\frac{\d}{\d\varphi}}$ satisfies
		$$
		2g=\L_\xi g = 2as d\varphi^2+2s^2db\cdot d\varphi + a ds\cdot d\varphi+sda\cdot d\varphi+sds\cdot db +2da\cdot ds=
		$$
		$$
		=\left(2as+2s^2\frac{\d b}{\d\varphi}+s\frac{\d a}{\d\varphi}\right)d\varphi^2+\left(2s^2\frac{\d b}{\d s}+a+s\frac{\d a}{\d s}+s \frac{\d b}{\d \varphi}+2\frac{\d a}{\d\varphi}\right)d\varphi\cdot ds+\left(s\frac{\d b}{\d s}+2\frac{\d a}{\d s}\right)ds^2.
		$$
		Therefore,
		\begin{equation}\label{Eq1}
		\left\{
		\begin{aligned}
		2as+2s^2\frac{\d b}{\d\varphi}+s\frac{\d a}{\d\varphi}&=2s^2 \\
		2s^2\frac{\d b}{\d s}+a+s\frac{\d a}{\d s}+s \frac{\d b}{\d \varphi}+2\frac{\d a}{\d\varphi}&=2s \\
		s\frac{\d b}{\d s}+2\frac{\d a}{\d s}&=2 \\
		\end{aligned}
		\right.
		\end{equation}
		The field $\xi$ is potential if and only if  
		\begin{equation}\label{Eq2}
		0=\iota_\xi g = d\left(bs^2d\varphi+asd\varphi+bsds+ads\right)=
		2bs+s^2\frac{\d b}{\d s} d\varphi+a+s\frac{\d a}{\d s}-s\frac{\d b}{\d \varphi}-\frac{\d a} {\d \v}.
		\end{equation}
		Denote 
		$$
		x=\frac{\d a}{\d \v},\ \ y=\frac{\d a}{\d s}, \ \  z=\frac{\d b}{\d \v}, \ \ w=\frac{\d b}{\d s}.
		$$
		Combining $\eqref{Eq1}$ and $\eqref{Eq2}$, we get the system of equation 
		$$
		\left\{
		\begin{aligned}
			sx+2s^2z&=2s^2-2as\\
			2x+sy+sz+2s^2w & = 2s-a\\
			2y+sw &=2 \\
			-x+sy-sz+s^2w &=-2bs. \\ 
		\end{aligned}
		\right.
		$$
		The solving of the system is the following 
		$$
		\left\{
		\begin{aligned}
		x&=a-2bs\\
		y& = \frac{a}{2s}-b+1\\
		z &=-\frac{3a}{2s}+b+1 \\
		w &=\frac{-a+2bs}{s^2}. \\ 
		\end{aligned}
		\right.
		$$
		We have 
		$$
		\frac{\d^2 a}{\d \v \d s}=\frac{\d x}{\d s}=\frac{\d a}{\d s}-2b-2s\frac{\d b}{\d s}=y-2b-2sw=\frac{a}{2s}-b+1-2b-\frac{-2a+4bs}{s}
		=-\frac{3a}{2s}-7b+1.
		$$
		On the other hand,
		$$
		\frac{\d^2 a}{\d \v \d s}=\frac{\d y}{\d \v}=\frac{1}{2s}\frac{\d a}{\d \v}-\frac{\d b}{\d \v}= \frac{x}{2s}-z=\frac{a}{2s}-b+\frac{3a}{2s}-b-1=\frac{2a}{s}-2b-1.
		$$
		Therefore,
		$$
		-\frac{3a}{2s}-7b+1=\frac{2a}{s}-2b-1.
		$$
		Hence,
		\begin{equation}\label{Eq4}
		7a+10sb-4s=0.
		\end{equation}	
		Then
		$$
		0=\frac{\d (7a+10sb-4s)}{\d \v}=7x+10sz=7a+14sb-15a+10sb+10s=-8a+24sb+10s.
		$$
		Combining this with \ref{Eq4}, we get
		$$
		a=\frac{49s}{62} \ \ \ \text{and} \ \ \ b=\frac{-19}{124}.
		$$
		Then
		$$
		0=\frac{\d b}{\d s}=w=\frac{-\frac{49s}{62}+\frac{-19s}{62}}{s^2}\ne 0.
		$$
		Therefore, there are not potential homothetic vector fields on $\left(S^1\times \R^{>0},g=s^2d\varphi^2+sds\cdot d\varphi+ds^2\right)$.
	\end{example}

	\section{Selfsimilar Hessian manifolds}\label{s3}

				\begin{defin}\label{d}
		A {\bfseries flat affine manifold} $(M,\nabla)$ is a differentiable manifold $M$ equipped with a flat torsion-free connection $\nabla$. Equivalently, it is a manifold equipped with an atlas such that all translation maps between charts are affine transformations (see e.g. \cite{FGH}). A {\bfseries radiant manifold} $(C,\nabla, \rho)$ is a flat affine manifold $(C,\nabla)$ endowed with a {\bfseries radiant vector field } $\rho$ i.e. a field satisfying
		\begin{equation}\label{3.1}
		\nabla \rho =\text{Id}.
		\end{equation}
		
		Equivalently, it is a manifold equipped with an atlas such that all translation maps between charts are linear transformations. In the corresponding local coordinates we have 
		$$
		\rho=\sum x^i \frac{\d}{\d x^i}
		$$ 
		(see e.g. \cite{Go}).  
	\end{defin} 

	\begin{defin}
		Let $(M,\nabla)$ be a flat affine manifold. If a Riemannian metric $g$ is locally expressed by a Hessian of a function
		$$
		\text{Hess}\ \varphi =\nabla d\varphi
		$$
		then $g$ is called a {\bfseries Hessian metric} and the triple $(M,\nabla, g)$ is called a {\bfseries Hessian manifold} (see e.g. \cite{shima}). More concretely, we have 
		$$
		g(X,Y)=XY(\varphi)-\nabla_X Y (\varphi).
		$$ 
	\end{defin}

	\begin{defin}
		A {\bfseries selfsimilar Hessian manifold} $(C,\nabla, g, \xi)$ is a Hessian manifold $(C,\nabla,g)$ endowed with a vector field $\xi$ such that $(C,g,\xi)$ is a selfsimilar manifold and the flow along $\xi$ preserves $\nabla$. If $\xi$ is complete then $C$ is called a {\bfseries global selfsimilar Hessian manifold}.    
	\end{defin}

\subsection{Selfsimilar Hessian manifolds with potential vector fields}
	\begin{proposition}[\cite{Go}]\label{pg}
	Let $(C,\nabla)$ be an affine flat manifold and $\xi$ a vector field such that the flow along $\xi$ preserves $\nabla$. Then 
	$$
	\nabla \nabla \xi =0.
	$$ 
	Equivalently, for any $X,Y \in TC$ we have 
	$$
	\nabla_X\nabla_Y \xi=\nabla_{\nabla_X Y} \xi.
	$$
\end{proposition}

\begin{lemma}\label{3.3}
		Let $\left(C,\nabla,g,\xi\right)$ be a selfsimilar Hessian manifold. Then the field $\xi$ is potential if and only if  
			\begin{equation} \label{z}
		g(X,\nabla_Y \xi) -g(Y, \nabla_X \xi)=0.
		\end{equation}
		
\end{lemma}

\begin{proof}
	By Theorem \ref{T1}, for any $X,Y\in TC$ we have
	$$
	\left(d\iota_\xi g\right)(X,Y) = X(g(\xi,Y))-Y(g(\xi,X))-g(\xi,[X,Y])=
	$$
	$$
	=XY\xi(\varphi)-X\nabla_Y \xi(\varphi) -YX\xi(\varphi)+Y\nabla_X\xi(\varphi)-[X,Y]\xi(\varphi)+\nabla_{[X,Y]} \xi (\varphi)=
	$$
	$$
	=-\left(X\nabla_Y \xi\right)(\varphi) +\left(Y\nabla_X \xi\right)(\varphi) +\left(\nabla_{[X,Y]} \xi \right)(\varphi).
	$$
	Combining this with the flatness of $\nabla$ we get
	$$
	\left(d\iota_\xi g\right)(X,Y)=-\left(X\nabla_Y \xi\right)(\varphi) +\left(Y\nabla_X \xi\right)(\varphi)+\nabla_X \nabla_Y \xi(\varphi)-\nabla_Y\nabla_X \xi(\varphi)=-g(X,\nabla_Y \xi) +g(Y, \nabla_X \xi).
	$$
	Therefore, the field $\xi$ is potential if and only if only \eqref{z} holds.
	
\end{proof}

According to Lemma \ref{3.3}, for any point $p\in C$ the linear operator $\nabla \xi|_p$ is self-adjoint with respect to $g|_p$. Therefore, $\nabla \xi|_p$ is diagonalizable. According to Proposition \ref{pg}, the operator $\nabla \xi : TC\to TC$ is $\nabla$ flat. Therefore, for any $p\in C$ the linear operator $\nabla_\xi |_p$, have the same eigenvalues $\lambda_1,\ldots,\lambda_k$. Let $V_i$ be a set of vectors $X\in TC$ satisfying $\nabla_X \xi =\lambda_i X$. Then $V_i$ is a $\nabla$ flat vector subspace of $TC$ and  

\begin{equation}\label{111}
TC=\bigoplus V_i,
\end{equation}

The affine manifold $C$ is locally isomorphic to $U=\prod U_i$, where the decomposition to the direct product is compatible \eqref{111}. 
\begin{proposition}
	Let $C,g,\xi, V_i$ be as above. Then for any $i\ne j$ we have $V_i\perp V_j$. As a consequence the orthogonal decomposition $g=\bigoplus V_i$ defines the decomposition of the metric $g=\sum g_i$. 
\end{proposition}
\begin{proof}
	According to \ref{3.3}, for any $X\in V_i,$ and $Y\in V_j$, where $i\ne j$, we have 
	$$
	0=g(X,\nabla_Y \xi) -g(Y, \nabla_X \xi)=\left(\lambda_2-\lambda_1\right) g(X,Y).
	$$
	Therefore, for any $i\ne j$ we have $V_i \perp V_j$. 
\end{proof}
\begin{proposition} \label{ppp}
	Let $C,g,\xi, V_i,U$ be as above. Suppose that $g=\nabla \alpha$ where $\alpha=d\varphi$ on $U$. The decomposition $V=\bigoplus V_i$ defines a decomposition $g=\sum g_i$, $\xi=\sum \xi_i, \alpha=\sum \alpha_i$. Then the following conditions are satisfies
	\begin{itemize}
		\item [(i)] \label{i} For any $j\ne i$, $X\in V_j$ we have $\L_X g_i=0$.
		\item [(ii)] \label{ii}For any $j\ne i$, $X\in V_j$ we have $\nabla_X \xi_i=0$. 
		\item [(iii)] \label{iii} For any $j\ne i$, $X\in V_j$ we have $\L_X \alpha_i=0$.
		\item [(iv)] \label{iv} For any $i$ we have $g_i=\nabla {\alpha_i}$.
		\item [(v)] \label{v} For any $i$ we have $\L_{\xi_i} g_i=2g_i$.  
	\end{itemize}
\end{proposition}
\begin{proof}
	(i) It is enough to check the identity 
	\begin{equation} \label{X}
	\left(\L_X g_i\right) (Y,Z)=0 
	\end{equation}
	for flat vector fields $Y,Z$. 
	
	First let us prove \eqref{X} for $X=f\tilde{X}$, where $\tilde X$ is a flat vector field and $f$ a function on $M$. 
	
	We have 
	\begin{multline*}
		\L_{f\tilde X} g_i (Y,Z)=f\tilde X(g_i(Y,Z))-g_i([f\tilde X,Y],Z)-g_i(Y,[f\tilde X,Z])= \\ 
		=f\tilde X(g_i(Y,Z))+Y(f) g_i(\tilde X, Z)+Z(f) g_i (\tilde X,Y)-fg_i([\tilde X,Y],Z)-fg_i(Y,[\tilde X,Z]).		
	\end{multline*}
	The vector fields $\tilde X, Y_i,Z_i$ are flat. Hence, 
	\begin{equation*}\label{355}
	\L_{f\tilde X} g_i (Y,Z)=f\tilde X(g_i(Y,Z))+Y(f) g_i(\tilde X, Z)+Z(f) g_i (\tilde X,Y)
	\end{equation*}
	Let $Y_i$ and $Z_i$ be projection of $X$ and $Y$ on $V_i$. Then
	$$
	\tilde X(g_i(X,Y))=\tilde X(g(Y_i,Z_i))=\tilde XY_iZ_i\varphi=Y_i\tilde XZ_i\varphi=Y_ig(\tilde X,Z_i)=0
	$$ 
	because $\tilde X\perp V_i$. Moreover,
	$$
	g_i(\tilde X,Z)=g_i(\tilde X, Y) =0.
	$$
	Thus, 
	$$
	\L_{f\tilde X} g_i (Y,Z)=0.
	$$
	Any vector field $X\in V_j$ is a sum of fields of the form $f\tilde X$, where $\tilde X$ is flat. We have proved \eqref{X} for $X=f\tilde{X}$. Therefore, for any $X\in V_j$.
	$$
	\L_X g_i=0.
	$$
	(ii) 	Since, $\nabla \xi$ is diagonalizable and $X\in V_j$, we have $\nabla_X\xi\in V_j$.  Hence, $$\sum_{k=1}^n \nabla_X \xi_k=\nabla_X {\xi}=X\in V_j.$$ 
	Since $\xi_k\in V_k$ and $V_k$ is $\nabla$-flat, we have $\nabla_X {\xi_k}\in V_k$, for any $k$.
Therefore, $\nabla_X \xi_i=0$.
	
	(iv) Let us prove that $\nabla \alpha_i\in \text{Sym}^2 V_i^*$. It is enough to prove that if $j\ne i$ and $X\in V_j$ then $\iota_X \nabla \alpha_i=0$. For any flat field $Y$, we have
	$$
	\left(\nabla \alpha_i\right)(Y,X) = Y\alpha_i(X)-\alpha_i\left(\nabla_Y X\right)=0
	$$ 
	because $X,\nabla_Y X\in V^j$. Thus $\nabla \alpha_i\in \text{Sym}^2 V_i^*$. Combining this with 
	$$
	\sum g_i= \nabla \alpha =\sum \nabla \alpha_i.
	$$
	we get that 
	$$
	\nabla \alpha_i=g_i.
	$$

	(iii) For any flat  $Y\in TU$ we have 
	\begin{equation}\label{re}
	\left(\L_X \alpha_i\right)(Y)=X(\alpha_i(Y))-\alpha_i([X,Y]).
	\end{equation}
	According to item (iv),
	$$
	g_i(X,Y)=\nabla \alpha_i\left(X,Y\right)=X\left(\alpha_i(Y)\right)-\alpha_i\left(\nabla_X Y\right).
	$$
	Combining this with \eqref{re}, we obtain 
	$$
	\left(\L_X \alpha_i\right)(Y)=g_i(X,Y)+\alpha_i\left(\nabla_X Y\right)-\alpha_i([X,Y])=g_i(X,Y)+\alpha_i\left(\nabla_Y X\right).
	$$
	Since $X\in V_j$, $g_i(X,Y)=0$ and $\nabla_Y X\in V_j$ because $V_j$ is $\nabla$-flat. Therefore,
	$$
		\left(\L_X \alpha_i\right)(Y)=g_i(X,Y)+\alpha_i\left(\nabla_Y X\right)=0.	$$

	(v) Let us prove that $\L_{\xi_i} g_i\in \text{Sym}^2 V_i^*$. It is enough to prove that if $j\ne i$ and $X\in V_j$ then $\iota_X \L_{\xi_i}g=0$. We have $\iota_Xg_i=0$ and $[\xi_i,X]=0$. For any field $Y$, we have 
	$$
	\L_{\xi_i} g_i(X,Y) = \xi_i g_i (X,Y)-g([\xi_i,X],Y)-g_i(X,[\xi_i, Y])=-g([\xi_i,X],Y).
	$$ 
	since $X\in V_j$. It follows from the item (ii) that $\xi_i\in V_i$. Therefore, $[\xi_i,X=0]$ and 
	$$
	\L_{\xi_i} g_i(X,Y)=0.
	$$
	Thus, $\L_{\xi_i} g_i\in \text{Sym}^2 V_i^*$. Combining this with 
	$$
	\sum 2g_i=2g=\L_\xi g = \L_{\xi_i} g_i,
	$$
	we get that
	$$
	\L_{\xi_i} g_i=2g_i.
	$$
\end{proof}

\begin{defin} \label{ddd}
	We say that a selfsimilar Hessian manifold $(C,\nabla, g, \xi)$ is {\bfseries radiant} if and only if there exists a radiant vector field $\rho$ and a constant $\lambda\ne0$ such that $\xi=\lambda\rho$. 
	Equivalently, there is a flat affine atlas such that in the corresponding local coordinates we have 
	$$
	\xi=\lambda\sum  x^i \frac{\d}{\d x^i}
	$$ 
	(see Definition \ref{d}).  
\end{defin}
  We say that a selfsimilar Hessian manifold $\left(U,\nabla,\xi,g\right)$ is a direct product of selfsimilar Hessian manifolds $\left(U_i,\nabla_i,\xi_i,g_i\right)$ if and only if
  $$U=\prod U_i, \ \ \ \nabla=\sum \nabla_i, \ \ \ g= \sum g_i, \ \ \ \xi=\sum \xi_i. $$
\begin{theorem} \label{t0}
		Let $\left(C,\nabla,\xi\right)$ be a selfsimilar Hessian manifold. Then $\xi$ is potential if and only if $\left(C,\nabla,\xi\right)$ is locally isomorphic to a direct product of radiant selfsimilar Hessian manifolds
\end{theorem}
\begin{proof}[Proof of Theorem \ref{l39}]
	Let $\xi$ is potential. The flat affine manifold $C$ is locally isomorphic to a product of flat affine manifolds $U=\prod U_i$. By item (i) of Proposition \ref{ppp}, $(U,g)$ is isometric to the direct product $\prod\left( U_i, g_i\right)$. By item (ii), if $i\ne j$ then $\xi_i$ does not depend on $U_j$. By items (iv) and (v), $\left(U_i,\nabla_i, g_i,\xi_i\right)$ is a selfsimilar Hessian manifold.       
	
	Let $(C,\nabla,g,\xi)$ is locally isomorphic to a direct product of selfsimilar Hessian manifolds $\left(U_i,\nabla_i, g_i,\xi_i\right)$. According to Lemma \ref{3.3}, we need to check \eqref{z}. There is a decomposition 
	$$
	TU=\bigoplus V_i
	$$
	compatible with the product $U=\prod U_i$. For any $X\in V_i, Y\in  V_j$, where $i\ne j$ we have $\nabla_X\xi=\lambda_i X\in V_i, \nabla_Y \xi=\lambda_j Y\in  V_j$. Therefore,
	$$
	g(X,\nabla_Y \xi) -g(Y, \nabla_X \xi)=0,
	$$
	since $V_i\perp V_j$. If $X,Y \in V_i$ then 
	$$
	g(X,\nabla_Y \xi) -g(Y, \nabla_X \xi)=g(X,\lambda_i Y)-g(\lambda_i X, Y)=0.
	$$
	Moreover, the term
	$
	\omega(X,Y)=g(X,\nabla_Y \xi) -g(Y, \nabla_X \xi)
	$
	is bilinear. Therefore, we have 
	$$
	g(X,\nabla_Y \xi) -g(Y, \nabla_X \xi)=0.
	$$
\end{proof}

\begin{proposition}[\cite{G-A}]\label{GA}
	Let $(C,\nabla, \rho)$ be a radiant manifold and $g$ a Hessian metric on $M$ with respect to $\nabla$. Then 
	$$
	\L_\rho  g =g +\nabla(\iota_\rho g).
	$$
\end{proposition}

\begin{proposition}\label{311}
	Let $(C,g,\xi)$ be a radiant Hessian manifold and $\lambda$ be as above. Then 
	$$
	g=\text{Hess}\left(\frac{g(\xi,\xi)}{4-2\lambda}\right).
	$$
\end{proposition}

\begin{proof}
	According to Proposition \ref{pp} we have 
	$$
	\iota_\xi g=d\left(\frac{g(\xi,\xi)}{2}\right).
	$$
	Combining this with Proposition \ref{GA}, we get 
	$$
	0=\lambda \L_\rho g -\lambda g -\lambda \nabla d\left(\iota_\rho g\right)=\L_\xi g -\lambda g -\nabla d\left(\iota_\xi g\right)=\left(2-\lambda\right) g -\nabla d \left(\frac{g(\xi,\xi)}{2}\right).
	$$
	Thus,
	$$
	g=\frac{1}{2-\lambda}\nabla\left(sds\right)=\nabla d\left(\frac{s^2}{4-2\lambda}\right)=\text{Hess}\left(\frac{s^2}{4-2\lambda}\right).
	$$
\end{proof}
	Since any selfsimilar Hessian manifold with a potential vector field is locally isomorphic to a direct product of radiant Hessian manifolds, the Proposition \ref{311} express a Hessian potential locally for any selfsimilar Hessian manifold with a potential vector field.

	\begin{rem}
		According to Theorem \ref{l39}, any global radiant Hessian manifold with a nowhere vanishing homothetic vector field admits a form $\left(M\times \R^{>0} , \nabla , g= s^2g_M+ds^2,s\ds\right)$ and, according to Proposition \ref{311}, the function $\frac{s^2}{4-2\lambda}$ is a potential of $g$. 
		Note that this construction is analogical to the construction of Sasakian manifolds. An odd dimensional Riemannian manifold $(M, g_M)$ is called Sasakian if the corresponding Riemannian cone $$
		{(\hat M,g) = (M\times \R^{>0}, s^2 g_M+ds^2)}
		$$ 
		is equipped with a dilatation-invariant complex structure, which makes $\hat M$ to a K\"ahler manifold (see \cite{Sasaki}). Then the function $s^2$ defines a K\"ahler potential of $g$ (see e.g. \cite{OV}). 
	\end{rem}

\subsection{The local structure of radiant Hessian manifolds}\label{s32}
\subsubsection{The case of a point where a homothetic vector field does not vanish}\label{321}
\begin{proposition}[\cite{Go}]\label{go}
	Let $(C,\nabla, \xi)$ be a radiant manifold. Then the flow along $\xi$ preserves the connection $\nabla$ .
\end{proposition}

\begin{lemma}\label{l0}
	Let $(M\times \R^{>0},g=s^2g_M+ds^2)$ be a Riemannian cone, $\pi: M\times \R^{>0} \to M$ a projection, $U$ an open subset of $M\times \R^{>0}$ such that the exists a connection $\nabla$ on $U$ such that $\left(U,\nabla, g,s\ds\right)$ is a selfsimilar Hessian manifold. Then we can extend $\nabla$ to a connection $\tilde\nabla$ on $\pi (U) \times \R^{>0}$ such that $(\pi (U) \times \R^{>0}, \tilde\nabla, g)$ is a global selfsimilar Hessian manifold.
\end{lemma}

\begin{proof}
	For any $q\in \R^{>0}$ consider a map $\lambda_q: M\times \R^{>0} \to M \times \R^{>0}$ defined by  $\lambda_q (x\times s)=x\times qs$. Then the flow along the vector field $s\ds$ during the time $\ln q$ coincides with the map $\lambda_q$. By Proposition \ref{go}, for any $x\in U$ and $q\in \R^{>0}$ satisfying $\lambda_q x \in U$ we have 
	$$
	(\lambda_q^* \nabla )(x) =\nabla (x).
	$$
	We can extend $\nabla$ to a connection $\tilde\nabla$ on $\pi(U)\times \R^{>0}$ such that 
	$$
	{\left(\pi(U) \times \R^{>0},\tilde\nabla,s\ds\right)}
	$$ is a radiant manifold using the same argument as in the proof of Lemma \ref{L3}. Since $\tilde g$ is selfsimilar
	$(\lambda_q(U), \nabla,\tilde g)$ is isometric to $(U,q^2\tilde g)$. 
	The restriction $\tilde g|_U$ is Hessian therefore $\tilde g|_{\lambda_q(U)}$ is Hessian. Thus, $\tilde g$ is a Hessian metric and ${(\pi(U) \times \R^{>0},\tilde\nabla,g,s\ds)}$ is a global selfsimilar Hessian manifold.
\end{proof}

\begin{proposition} \label{39} 
		Any selfsimilar Hessian manifold is locally isomorphic to a global selfsimilar Hessian manifolds.

\end{proposition}

\begin{proof}

	Let $(C,\nabla,g,\xi)$ be a selfsimilar Hessian manifold. By Lemma \ref{L3}, the selfsimilar manifold $\left(C,g,\xi\right)$ is locally isomorphic to a part of a global selfsimilar manifold $\left(M\times\R^{>0},\tilde g,s\ds\right)$. By Lemma \ref{l0}, the connection $\nabla$ on $C\subset M\times\R^{>0}$ can be extended to a connection $\tilde \nabla$ on $M\times \R^{>0}$ such that $\left(M\times \R^{>0},\nabla,g,s\ds\right)$ is a global selfsimilar Hessian manifold. 
\end{proof}

\begin{defin}
	A cone $V\subset \R^{n+1}$ is an open subset of $\R^{n+1}$ such that if ${x=(x^1,\ldots ,x^{n+1})\in V}$ then for any $t\in \R^{>0}$ we have $tx=(tx^1,\ldots,tx^{n+1}) \in V$. A cone is called {\bfseries regular} if it does not contain any full straight line. 
\end{defin}

    Let $V\subset \R^{n+1}$ be a regular cone. Then we can obtain that $V\in \R^{n}\times \R^{>0}$. Further, we will consider any regular cone $V$ as a subset $V\subset \R^{n}\times\R^{>0}$.  
    
\begin{defin}
    A {\bfseries conical Hessian domain} $(V,g)$ is a regular cone $V\subset\R^{n}\times\R^{>0}$ endowed with a standard affine connection from $\R^{n+1}$ and a Hessian metric $g$ such that 
    $$
    \L_\rho  g = 2\lambda^{-1}g,
    $$
    where $\lambda\in \R\setminus\{0\}$ and 
    $$
    \rho= \sum_{i=1}^{n+1} x^i \frac{\partial}{\partial x^i}
    $$
    is the {\bfseries radiant vector field}.
\end{defin}
	Let $\nabla$ be a standard affine connection on $\R^{n+1}$  then the field $\xi=\lambda\rho$ satisfies $\L_\xi g =2g$ and $\nabla \xi =\lambda \text{Id}.$ Hence, $\left(V,\nabla,g,\xi\right)$ is a radiant selfsimilar Hessian manifold. 

\begin{proposition}\label{316}
A selfsimilar Hessian manifold with a nowhere vanishing homothetic field is locally isomorphic to a conical Hessian domain.
\end{proposition}

\begin{proof}
	Let $\left(C,\nabla,g,\xi\right)$ be a selfsimilar manifold and $\xi$ is nowhere vanishing. The collection $(C,\nabla,\xi)$ is locally isomorphic to a collection $(U,\nabla', \lambda\sum_{i=1}^{n+1} x^i\dxi)$, where 
	${U\subset\R^{n+1}}$, 
	$\nabla'$ is a standard flat connection on $\R^{n+1}$ and 
	$x^1,\ldots,x^{n+1}$ are coordinates on $\R^{n+1}$ (see Definition \ref{ddd}). The neighborhood $U\in \R^{n+1}$ contains a convex bounded domain $U'$. Then 
	$$
	V:=\bigcup_{t\in \R^{>0}} tU'.
	$$ 
	is a regular cone. By Lemma \ref{L3}, we can extend $g$ to a selfsimilar metric $\tilde g$ on $V$. 
	The collection $\left(U,\nabla', \tilde g\right)$ is isomorphic to $\left(tU,\nabla', t^{-2}\tilde g\right)$. Hence, the metric $t^{-2}\tilde g|_{tU}$ is Hessian. Therefore, $g|_{tU}$ is Hessian. The same is true for any $t\in \R^{>0}$. Thus, $\tilde g$ is a Hessian metric on $V$ and $(V,\tilde g)$ is an affine Hessian cone.
	
\end{proof}

\begin{cor}
	A selfsimilar manifold is locally isomorphic to an affine Hessian cone near a point where the homothetic vector field does not vanish. 
\end{cor}

 \begin{proposition}\label{317}
	Let $(C,g)$ be a conical Hessian domain, $\lambda$ as above,
	$$
	\varphi = \frac{g(\xi,\xi)}{4-2\lambda}.
	$$
	and
	$$
	\psi=\varphi\left/\left(x^{n+1}\right)^{2\lambda^{-1}}\right.. 
	$$
	Then
	$$
	g=\text{Hess}\ \varphi = \text{Hess} \left(\left(x^{n+1}\right)^{2\lambda^{-1}} \psi\right)
	$$
	and $\psi$ is constant along the radiant vector field 
	$$
	\rho =
\sum_{i=1}^{n+1} x^i \dxi. 
	$$ 
	
\end{proposition}
\begin{proof} 
	By Proposition \ref{311}, $g=\text{Hess} \ \varphi$. We have 
	$$
	\rho\left(g(\rho,\rho)\right) =\L_\rho g (\rho,\rho)=2\lambda^{-1}g (\xi,\xi).
	$$
	Therefore, $\xi \varphi =\lambda^{-1}\varphi$. Using this, we get 
	$$
	\rho \left(\frac{\varphi}{\left(x^{n+1}\right)^{2\lambda^{-1}}}\right)=
	\frac{2\lambda^{-1}\varphi \left(x^{n+1}\right)^{2\lambda^{-1}}-\varphi \left(\sum_{i=1}^{n+1} x^i\dxi\right) \left(\left(x^{n+1}\right)^{2\lambda^{-1}}\right)}{\left(x^{n+1}\right)^{4\lambda^{-1}}}=0.
	$$
	
\end{proof}

	  \begin{proposition}\label{l33}
	 	Let $V\subset \R^{n}\times\R^{>0}$, ${U=V\cap \{x^{n+1}=1\}}$, $\rho=\sum_{i=1}^{n+1}x^i\dxi$, $\psi$ be a constant along $\rho$ function, $\lambda\in \R \setminus\{0,2\}$. Then the bilinear form  $\text{Hess} \left(\left(x^{n+1}\right)^{2\lambda^{-1}} \psi\right)$ is positive definite  if and only if at any point we have
	 	\begin{equation} \label{eq}
	 	\left(4\lambda^{-2}-2\lambda^{-1}\right)\psi>0  \ \ \ \text{and} \ \ \ X^2  \psi>\max \left(\left(2-4\lambda^{-1}\right)X\left(\psi\right)-\left(4\lambda^{-2}-2\lambda^{-1}\right)\psi ,0 \right),
	 	\end{equation}
	\end{proposition}
	
	\begin{proof}
		the bilinear form $g$ is positive definite if and only if for any flat nonzero field $X\in TU$ we have 
		$$
		g(X,X)>0, \ \  g(\rho,\rho)=0 , \ \ \text{and} \ \  g(X+\rho,X+\rho)>0.
		$$
		The first inequality is satisfied if and only if $X^2\psi>0$. 
		We have,
		$$
		g(\rho,\rho)=\rho^2\left(\left(x^{n+1}\right)^{2\lambda^{-1}} \psi\right)-\nabla_\rho \rho\left(\left(x^{n+1}\right)^{2\lambda^{-1}} \psi\right).
		$$
		We have $\nabla_\rho =\rho$, $\rho \psi=0$ and $\rho\left(x^{n+1}\right)=x^{n+1}$. Hence,
		$$
		g(\rho,\rho)=\left(4\lambda^{-2}-2\lambda^{-1}\right)\left(x^{n+1}\right)^{2\lambda^{-1}}\psi.
		$$  
		Therefore, $g(\rho,\rho)>0$ if and only if
		$$
		\left(4\lambda^{-2}-2\lambda^{-1}\right)\psi>0.
		$$

		We have
		\begin{equation}\label{36}
		g(X,\rho)=X\rho\left(\left(x^{n+1}\right)^{2\lambda^{-1}} \psi\right)-\nabla_X \rho\left(\left(x^{n+1}\right)^{2\lambda^{-1}} \psi\right)=X\rho\left(\left(x^{n+1}\right)^{2\lambda^{-1}} \psi\right)- X\left(\left(x^{n+1}\right)^{2\lambda^{-1}} \psi\right).
		\end{equation}
		Since $\rho(\psi)=0$ and $\rho\left(x^{n+1}\right)=x^{n+1}$, we have
		$$
		\rho\left(\left(x^{n+1}\right)^{2\lambda^{-1}}\psi\right)=\psi\rho\left(\left(x^{n+1}\right)^{2\lambda^{-1}}\right)=\psi \sum_{i=1}^{n+1}x^i\dxi \left(\left(x^{n+1}\right)^{2\lambda^{-1}}\psi\right)=2\lambda^{-1}\left(x^{n+1}\right)^{2\lambda^{-1}}\psi.
		$$ 
		Combining this with \ref{36}, we get
		$$
		g(X,\rho)=2X\left(\lambda^{-1}\left(x^{n+1}\right)^{2\lambda^{-1}}\psi\right)-X\left(\left(x^{n+1}\right)^{2\lambda^{-1}}\psi\right)=\left(2\lambda^{-1}-1\right)\left(x^{n+1}\right)^{2\lambda^{-1}}X\left(\psi\right).
		$$
		Therefore,
		$$
		g(X+\rho,X+\rho)=g(X,X)+2g(X,\rho)+g(\rho,\rho)=
		$$
		$$
		=\left(x^{n+1}\right)^{2\lambda^{-1}}\left(X^2+\left(4\lambda^{-1}-2\right)X+4\lambda^{-2}-2\lambda^{-1}\right)(\psi)>0.
		$$
		Since $x^{n+1}\in \R^{>0}$, we have $g(X+\rho,X+\rho)>0$ if and only if
		$$
		\left(X^2+\left(4\lambda^{-1}-2\right)X+4\lambda^{-2}-2\lambda^{-1}\right)(\psi)>0.
		$$
		Thus, $g$ g is positive definite if and only if \eqref{eq} holds.
		
	\end{proof}

\begin{theorem}\label{TTT}
		Let $V\subset \R^{n}\times \R^{>0}$ be a regular cone then $(V,g)$ is a affine Hessian cone if and only if 
		$$
		g=\text{Hess} \left(\left(x^{n+1}\right)^{2\lambda^{-1}}\psi\right),
		$$ 
		where 
        $\psi$ is a function constant along the radiant vector field such that
		\begin{equation*}
	 \left(4\lambda^{-2}-2\lambda^{-1}\right)\psi>0 \ \ \ \text{and} \ \ \ X^2  \psi>\max \left(\left(2-4\lambda^{-1}\right)X\left(\psi\right)-\left(4\lambda^{-2}-2\lambda^{-1}\right)\psi ,0 \right), 
		\end{equation*}
		for any nonzero flat vector field $X=\sum_{i=1}^n b_i \dxi\in TU$, $b_i\in \R$ where $U=V\cup\{x^{n+1}=1\}.$
	\end{theorem}
	\begin{proof}
		Theorem \ref{TTT} follows from propositions \ref{317} and \ref{l33}.  
	\end{proof}
\begin{cor}\label{cor}
	There exists a radiant Hessian manifold with any $\lambda \ne 0,2$.  
\end{cor}
	\subsubsection{The case of a point where a homothetic vector field vanishes}\label{322}
	\begin{proposition}\label{pppp}
		Let $(C,\nabla,g,\xi)$ be a radiant selfsimilar Hessian manifold and $p\in C$ a point such that $\xi_p=0$. Then, near the point $p$, $(C,\nabla,g,\xi)$ is isomorphic to a part of selfsimilar Hessian manifold of the form $\left(\R^{n+1}, g,\lambda\sum_{i=1}^{n+1} x^i\dxi\right)$ with a standard affine connection. 
	\end{proposition}
	\begin{proof}
		By Definition \ref{ddd}, there exists a constant $\lambda$ such that $\rho=\lambda^{-1}\xi$ is a radiant vector field. Then an affine manifold $C$ with a field $\rho$ is isomorphic to a star neighborhood $U$ of $\R^{n+1}$ with a field $\sum_{i=1}^{n+1} x^i\dxi$ near a point $p$. Hence, $\xi$ is isomorphic to $\lambda \sum_{i=1}^{n+1} x^i\dxi$. Using Lemma \ref{L3}, for the manifold $\R^{n+1}\setminus \{0\}=S^{n}\times \R^{>0}$, we can be extended $g|_U$ to a selfsimilar metric $\tilde g$ on $\R^{n+1}$. The metric $\tilde g$ is Hessian by the same argument as in the prove of Proposition \ref{316}.
	\end{proof}
	\begin{proposition}\label{324}
		Let $\left(\R^{n+1}, g,\xi=\lambda\sum_{i=1}^{n+1} x^i\dxi\right)$ be a selfsimilar Hessian manifold,  
		$$
		r=\left(\sum_{i=1}^{n+1}\left(x^i\right)^2\right)^{\frac{1}{2}}, \ \ \ \text{and} \ \ \ \psi=\frac{g(\xi,\xi)}{(4-2\lambda)r^{2\lambda^{-1}}}.
		$$ 
		Then then $\psi$
		is constant along the radiant vector field $\rho=r\dr$ and 
		$$g=\text{Hess} \left(r^{2\lambda^{-1}} \psi\right).$$
	\end{proposition}
 	\begin{proof}
 		By Proposition \ref{311} 
 		$
 		g=g=\text{Hess} \left(r^{2\lambda^{-1}} \psi\right).
 		$
 		The function $\psi$ is constant along $\xi$ by the same argument as in Proposition \ref{317}.
 	\end{proof}
 	\begin{theorem} \label{t1}
 		Let $\left(C,\nabla,g,\xi\right)$ be a radiant Hessian manifold with vanishing at a point homothetic field $\xi$. Then $\nabla \xi =\text{Id}$.  
 	\end{theorem} \begin{proof}
 	According to Proposition \ref{pppp}, near a point where the homothetic vector field vanishes $(C,\nabla,g,\xi)$ is locally isomorphic to $\left(\R^{n+1}, \tilde g,\tilde \xi=\lambda \sum_{i=1}^{n+1} x^i\dxi\right)$. Let $r$ be as above. According to Proposition \ref{324}, there exist a constant along $\rho$ function $\psi$ such that $$g=\text{Hess} \frac{r^{2\lambda^{-1}} \psi}{4-2\lambda}.$$ We have 
 	$$
 	g\left(\dr,\dr \right)=\frac{\d^2}{\d r^2}  \frac{r^{2\lambda^{-1}} \psi}{4\lambda^{-1}-2} - \nabla_\dr \dr  \frac{r^{2\lambda^{-1}} \psi}{4\lambda^{-1}-2}=\frac{\d^2}{\d r^2}  \frac{r^{2\lambda^{-1}} \psi}{4\lambda^{-1}-2}=4\lambda^{-1} {r^{2\lambda^{-1}-2} \psi}.
 	$$  
 	Since $g\left(\dr,\dr\right)$ is defined and does not equal $0$, we have ${2\lambda^{-1}-2}=0$. Therefore, $\lambda=1$ and $\nabla \xi=\text{Id}$ near a point $p$. By Proposition \ref{pg}, $\nabla\nabla \xi=0$. Therefore, $\nabla \xi=\text{Id}$ on $C$.
 \end{proof}
\begin{lemma}\label{l37}
	Let $\left(C,\nabla,g,\xi\right)$ be a selfsimilar Hessian manifold such that $\nabla \xi=\text{Id}$. Then $\nabla_\xi g= 0$.
\end{lemma}
 \begin{proof}
 	According to Theorem \ref{t1}, we have $\nabla \xi=\text{Id}$. For any $X,Y \in T\R^n$ we have
 	$$
 	\left(\nabla_\rho g\right)=\rho \left(g(X,Y)\right)- g(\nabla_\rho X,Y)-g(X,\nabla_\rho Y)
 	$$
 	Combing this with
 	$$
 	2g(X,Y)=\L_\rho g(X,Y)= \rho\left(g(X,Y)\right)- g([\rho, X],Y)-g(X,[\rho, Y]),
 	$$
 	we get
 	\begin{multline*}
 	\nabla_\rho g(X,Y)=2g(X,Y)+ g([\rho, X],Y)+g(X,[\rho, Y]) -  g(\nabla_\rho X,Y)-g(X,\nabla_\rho Y)= \\
 	=2g(X,Y)-g(\nabla_X \xi,Y)-g(X,\nabla_Y \xi)=0.
 	\end{multline*}

 \end{proof}


 	 \begin{proposition}\label{326}
 	Let $\nabla$ be a standard affine connection on $\R^{n+1}$. If $\left(\R^{n+1},\nabla,g,\lambda\sum_{i=1}^{n+1}x_i\dxi\right)$ is selfsimilar Hessian manifold then the metric $g$ is $\nabla$-flat. 
 \end{proposition}
	Note that the Riemannian flatness of $g$ follows from Theorem \ref{T_1}. However, the $\nabla$-flatness of $g$ a priori does not follow from the Riemannian flatness. In fact, for an arbitrary Hessian (not selfsimilar) manifold $\left(M,\nabla,g\right)$, the metric $g$ may be Riemannian flat but not $\nabla$-flat.
	\begin{example}
		Consider an affine space $\R^n$ with a Hessian metric 
		$$
		g=\text{Hess}\left(\sum_{i=1}^n e^{x^i}\right)=\sum_{i=1}^n e^{x^i}\left(dx^i\right)^2.
		$$
		The set 
		$
		\left\{y^i=2e^{x^i/2}\right\}
		$
		is a coordinate system. Then
		$$
		g=\sum_{i=1}^n e^{x^i}\left(dx^i\right)^2=\sum_{i=1}^n \left(d\left(2e^{x^i/2}\right)\right)^2=\sum_{i=1}^n \left(dy^i\right)^2.
		$$
		Therefore, the Riemannian curvature of $g$ is equal to $0$. That is, $g$ is Riemannian flat.
	\end{example}

\begin{proof}[Proof of Proposition \ref{326}]
 According to Theorem \ref{t1}  we have $\lambda=1$. According to Lemma \ref{l37}, $\nabla_{\sum_{i=1}^{n+1} x^i\dxi} g = 0.$ Therefore, for any $x\in\R^{n+1}$ the metric $g$ is constant on the orbit of $x$ along the field $\xi$. That is, $g$ is constant on the set $\gamma_x=\{tx\in \R^{n+1}| t\in \R^{>0}\}$ i.e. if $g=\sum c_{i,j}dx^idx^j$ then all function $c_{i,j}$ are constant on $\gamma_x$. Moreover, $g$ is continuous. Therefore, $g$ is constant on the closing $\overline \gamma_x=\gamma_x\cup\{0\}$. Hence, for any $x\in \R^{n+1}$ we have $g(0)=g(x)$. Therefore $g$ is $\nabla$-flat.
\end{proof}

\begin{proof}[Proof of Theorem \ref{T4}]
	(i) Combining propositions \ref{pppp} and \ref{326}, we get that, near the point $p$ the radiant selfsimilar Hessian manifolds is isomorphic to $\left(\R^{n+1},\nabla,g,\rho\dxi\right)$, where $\nabla$ is a standard flat connection on $\R^n$, $\rho$ is the radiant vector field, and $g$ is $\nabla$-flat. Then in a proper affine coordinate system $\left(x^1\ldots x^n\right)$ we have $g=\sum_{i=1}^{n+1}\left(dx^i\right)^2$. The radiant vector field has the same form in any flat coordinates $\rho=\sum_{i=1}^{n+1} x^i\dxi$.

	(ii) It follows from Proposition \ref{316} and Theorem \ref{TTT}.
\end{proof}

		\paragraph*{Acknowledgements.} I would like to thank M. Verbitsky for fruitful discussions, and {D.V. Alekseevsky}, for his useful comments and help with preparation of the paper.

\end{document}